\documentclass[11pt]{article}
\usepackage{epigamath}

\usepackage[notext]{kpfonts}
\usepackage{baskervald}

\setpapertype{A4}

\usepackage[english]{babel}

\usepackage[dvips]{graphicx}     

\title{\vspace{-0.5cm}Prime Fano threefolds of genus~12 with a $\mathbb{G}_\mathrm{m}$-action \\ and their automorphisms}
\titlemark{Prime Fano threefolds of genus~12 with a $\mathbb{G}_\mathrm{m}$-action}
\author{\vspace{0cm} Alexander Kuznetsov and Yuri Prokhorov}
\authoraddresses{
\authordata{Alexander Kuznetsov}{\firstname{Alexander} \lastname{Kuznetsov}\\
\institution{Steklov Mathematical Institute of Russian Academy of Sciences, Moscow, Russia \newline
Interdisciplinary Scientific Center J.-V. Poncelet (CNRS UMI 2615), Moscow, Russia
\newline
Laboratory of Algebraic Geometry, National Research University Higher School of Economics, Moscow, Russia}\\
\email{akuznet@mi.ras.ru}}\\
\authordata{Yuri Prokhorov}{\firstname{Yuri}
\lastname{Prokhorov}\\
\institution{Steklov Mathematical Institute of Russian Academy of Sciences, Moscow, Russia
\newline
Laboratory of Algebraic Geometry, National Research University Higher School of Economics, Moscow
\newline
Department of Algebra, Moscow State University, Moscow}\\
\email{prokhoro@mi.ras.ru}}
}
\authormark{A. Kuznetsov and Y. Prokhorov}
\date{\vspace{-5ex}} 
\journal{\'Epijournal de G\'eom\'etrie Alg\'ebrique} 
\acceptation{Received by the Editors on January 4, 2018, and in final form
on April 16, 2018. \\ Accepted on April 30, 2018.}

\newcommand{\articlereference}{Volume 2 (2018), Article Nr. 3}

 \usepackage[all]{xy}

\allowdisplaybreaks

\numberwithin{equation}{numsection}

\newtheorem{theorem}[equation]{Theorem}

\newtheorem{proposition}[equation]{Proposition}
\newtheorem{lemma}[equation]{Lemma}
\newtheorem{corollary}[equation]{Corollary}

\newtheorem{remark}[equation]{Remark}

\usepackage{tikz,upgreek}

\DeclareMathOperator{\GL}{\mathrm{GL}}
\DeclareMathOperator{\PGL}{\mathrm{PGL}}

\newcommand{\Aut}{\operatorname{Aut}}

\newcommand{\Pic}{\operatorname{Pic}}
\newcommand{\PP}{\mathbb{P}}
\newcommand{\ZZ}{\mathbb{Z}}
\newcommand{\CC}{\mathbb{C}}
\newcommand{\GG}{\mathbb{G}}

\newcommand{\cQ}{{\mathscr{Q}}}
\newcommand{\cX}{{\mathscr{X}}}
\newcommand{\cC}{{\mathscr{C}}}
\newcommand{\NNN}{{\mathscr{N}}}
\newcommand{\OOO}{{\mathscr{O}}}

\newcommand{\cO}{\mathscr{O}}

\newcommand{\xref}[1]{\textup{\ref{#1}}}
\newcommand{\nr}{{\mathrm{nr}}}
\newcommand{\MU}{{\mathrm{MU}}}
\newcommand{\m}{{\mathrm{m}}}
\newcommand{\g}{{\mathrm{g}}}

\begin{document}


\maketitle



\begin{prelims}


\def\abstractname{Abstract}
\abstract{We give an explicit construction of prime Fano threefolds of genus~12 with a $\mathbb{G}_\mathrm{m}$-action, 
describe their isomorphism classes and automorphism groups.}



\vspace{0.05cm}

\languagesection{Fran\c{c}ais}{%

\textbf{Titre. Solides de Fano primitifs de genre 12 avec action de $\mathbb{G}_\mathrm{m}$ et leurs automorphismes} \commentskip \textbf{R\'esum\'e.} 
Nous donnons une construction explicite des
solides de Fano primitifs (ceux dont le fibr\'e canonique engendre le groupe de Picard) de genre~12, \'equip\'es d'une action de $\mathbb{G}_{\mathrm{m}}$, et nous d\'ecrivons leurs classes d'isomor\-phisme et leurs groupes d'automorphismes.}

\end{prelims}


\newpage

\setcounter{tocdepth}{1}
\tableofcontents

\section{Introduction} 

We work over the field $\CC$ of complex numbers.
Recall that a \emph{prime Fano threefold} is a smooth projective variety of dimension 3 with~$\Pic(X) \cong \mathbb{Z} K_X$ and $-K_X$ ample.
The genus $\g(X)$ is defined by the equality $(-K_X)^3 = 2\g(X)- 2$.
We refer to~\cite{Iskovskikh-Prokhorov-1999} for the general theory of Fano varieties and their classification in dimension~3.

In this paper we discuss prime Fano threefolds of genus 12, which means $(-K_X)^3 = 22$.
Specifically, we give an explicit description of those of them that admit a faithful action of a one-dimensional torus $\GG_\m$.
Thus, this paper complements to~\cite{Prokhorov-1990c}, \cite[\hbox{\S\S5.3--5.4}]{Kuznetsov-Prokhorov-Shramov}, 
\cite{Dinew-Kapustka-Kapustka-2015}, and~\cite{D08,D17,CS18}.
To state the main result, we introduce some notation.

Consider the projective four-space $\PP^4$ with a $\GG_\m$-action with weights $(0,1,3,5,6)$ and let~$y_0,y_1,y_3,y_5,y_6$ be the corresponding coordinates.
Consider the map $\PP^1 \to \PP^4$ 
\begin{equation}
\label{eq:gamma}
(t_0 : t_1) \longmapsto (y_0 : y_1 : y_3 : y_5 : y_6) = (t_0^6 : t_0^5t_1 : t_0^3t_1^3 : t_0t_1^5 : t_1^6),
\end{equation} 
and let $\Gamma \subset \PP^4$ be its image.
This is a smooth rational sextic curve.

Consider the pencil of quadrics in $\PP^4$ generated by 
\begin{equation}
\label{eq:q0-q8}
Q_0 = \{ y_0y_6- y_3^2 = 0 \}
\qquad\text{and}\qquad 
Q_\infty = \{ y_3^2- y_1y_5 = 0 \}.
\end{equation}
For $u = (u_0:u_1) \in \PP^1$ denote by $Q_u = u_0Q_0 + u_1Q_\infty$ the corresponding quadric in the pencil.
All quadrics~$Q_u$ pass through $\Gamma$ and $Q_u$ is smooth if and only if $u \in \PP^1 \setminus \{0,1,\infty\}$.

\begin{theorem}
\label{main}
There is a natural bijection between 
isomorphism classes of prime Fano threefolds of genus~$12$ with a faithful\/ $\GG_\m$-action, and 
the set $\PP^1 \setminus \{0,1,\infty\}$ of smooth quadrics $Q_u$ passing through the curve $\Gamma$.

Set $u = u_0/u_1$.
If $X^\m(u)$ denotes the threefold corresponding to the quadric~$Q_u$, then
\begin{equation*}
\Aut(X^\m(u)) \cong 
\begin{cases}
\PGL_2, & \text{if $u = -1/4$}, \\
\GG_\m \rtimes \ZZ/2\ZZ, & \text{if $u \in \PP^1 \setminus \{0,1,-1/4,\infty\}$},
\end{cases}
\end{equation*}
where the action of $\ZZ/2\ZZ$ on $\GG_\m$ in the semidirect product is given by the inversion.

In particular, $X^\m(-1/4) \cong X^\MU$ is the Mukai--Umemura threefold.
\end{theorem}

The proof of the theorem takes the rest of the paper.
We compute the automorphisms in Corollary~\ref{corollary:automorphisms}, 
establish the bijection of isomorphism classes in Corollary~\ref{corollary:bijection},
and identify the Mukai--Umemura threefold in Proposition~\ref{proposition:q-mu}.

\section{A birational transformation}
\label{section:transformation}

The proof of Theorem~\ref{main} is based on a \emph{Sarkisov link} described in Theorem~\ref{theorem-diagram} below.
To state it we need to remind the notion of quadratical normality for projective curves.

Recall that a curve in a projective space is called \emph{quadratically normal} 
if restrictions of quadrics form a complete linear system on the curve.
We will need the following evident observation.

\begin{remark}
\label{remark:gamma-qn}{\rm
A rational sextic curve $\Gamma \subset \PP^4$ is quadratically normal if and 
only if the linear system of quadrics passing through $\Gamma$ is one-dimensional. 
Indeed, this follows immediately from the exact sequence
\begin{equation*}
0 \to H^0(\PP^4,I_\Gamma(2)) \to H^0(\PP^4,\cO_{\PP^4}(2)) \to H^0(\Gamma,\cO_\Gamma(12))
\end{equation*}
since $\dim H^0(\PP^4,\cO_{\PP^4}(2)) = 15$ and $\dim H^0(\Gamma,\cO_\Gamma(12)) = 13$.
In particular, a quadratically normal rational sextic curve in $\PP^4$ is not contained in a hyperplane.}
\end{remark}

The next result is the base of our construction.

\begin{theorem}[\cite{Takeuchi-1989}, \cite{Iskovskikh-Prokhorov-1999}]
\label{theorem-diagram} 
Let $X$ be a prime Fano threefold of genus~$12$ and let $C$ be a smooth conic on~$X$. 
Then there exists the following commutative diagram 
of birational maps
\begin{equation}
\label{diagram}
\vcenter{\xymatrix{
& X' \ar[dl]_{\sigma_X} \ar@{-->}[rr]^{\chi}
&& Q'\ar[dr]^{\sigma_Q}
\\
X \ar@{-->}^{\xi}[rrrr]
&&&& Q
}}
\end{equation}
where 
\begin{itemize}
\item 
$Q$ is a smooth quadric in $\PP^4$, 
\item 
$\sigma_Q$ is the blow up of a smooth rational quadratically normal sextic curve $\Gamma\subset Q$,
\item 
$\sigma_X$ is the blow up of $C$, 
\end{itemize}
and $\chi$ is a flop.

Furthermore, let $H_X$ and $H_Q$ be the ample generators of the Picard groups $\Pic(X)$ and~$\Pic(Q)$, respectively, 
and denote $H'_X := \sigma_X^*H_X$ and $H'_Q := \sigma_Q^*H_Q$.
Let $E_C:=\sigma_X^{-1}(C)$ and $E_\Gamma:=\sigma_Q^{-1}(\Gamma)$ be the exceptional divisors of~$\sigma_X$ and~$\sigma_Q$, respectively. 
Then
\begin{enumerate}
\item[\rm (i)]
\label{theorem-diagram-1}
The map $\sigma_Q \circ \chi \colon X' \dashrightarrow Q\subset \PP^4$ is given by the linear system $|H'_{X}- 2E_C|$ and contracts the unique divisor from $|2H'_X- 5E_C|$. 
\item[\rm (ii)]
\label{theorem-diagram-2}
The map $\sigma_X \circ \chi^{-1} \colon Q' \dashrightarrow X\subset \PP^{13}$ is given by the linear system $|5H'_Q- 2E_\Gamma|$ and contracts the unique divisor from $|2H'_Q- E_\Gamma|$.
\end{enumerate}
\end{theorem}

\begin{proof}
Let $X'$ be the blowup of $X$ along $C$.
The linear system $|-K_{X'}|$ is base point free and defines a morphism that does not contract divisors by Lemma~\ref{lemma:kxprime} below.
Therefore, \cite{Takeuchi-1989} or~\cite[Theorem~4.4.11]{Iskovskikh-Prokhorov-1999} applies 
and gives the diagram, and most of the details of the theorem.
The only thing left is to show that the curve $\Gamma$ is quadratically normal.

The argument of Remark~\ref{remark:gamma-qn} shows that
the curve $\Gamma$ is quadratically normal if and only if the linear system $|2H'_Q- E_\Gamma|$ is zero-dimensional.
On the other hand, the map $\chi^{-1}$ identifies this linear system with~$|2(H'_X- 2E_C)- (2H'_X- 5E_C)| = |E_C|$,
which consists of a single divisor, hence the claim.
\hfill $\Box$
\end{proof}

\begin{lemma}
\label{lemma:kxprime}
Let $X'$ be the blowup of $X$ along a smooth conic $C$.
The linear system $|-K_{X'}|$ is base point free and defines a morphism that does not contract divisors.
\end{lemma}
\begin{proof} 
Recall that $H_X =-K_X$ is very ample,
\begin{equation}
\label{eq:14}
\dim H^0(X,\cO_X(-K_X)) = 14,
\end{equation} 
see~\cite[Corollary 2.1.14(ii)]{Iskovskikh-Prokhorov-1999}, and so $|-K_X|$ embeds $X$ into $\PP^{13}$.
Note that $X \subset \PP^{13}$ does not contain planes (because~$\Pic(X)=\ZZ H_X$)
and is an intersection of quadrics (see~\cite[Proposition IV.1.3 and Theorem II.3.4]{Iskovskikh-1980-Anticanonical}).

Note that $|-K_{X'}| = |H'_X- E_C|$ is the strict transform of the linear system of hyperplane sections of $X$ passing through the conic $C$.
Their scheme-theoretic intersection is the intersection of the linear span of $C$ with $X$.
Since $X$ is an intersection of quadrics and does not contain planes, this is equal to $C$.
Therefore, the linear system $|-K_{X'}|$ on~$X'$ is base point free
(moreover, it follows that the morphism defined by this linear system is birational, although we do not need this).
In particular, $-K_{X'} \cdot \gamma \ge 0$ for any curve $\gamma$ on~$X'$.

Assume $\gamma$ is contracted by $|-K_{X'}|$, so that \mbox{$-K_{X'} \cdot \gamma = 0$}.
Then $H'_X \cdot \gamma \ne 0$, since $\gamma$ is an effective curve, so since $H'_X$ is nef we have $H'_X \cdot \gamma > 0$.
Therefore $E_C \cdot \gamma > 0$ and so
\begin{equation*}
(H'_X- 2E_C) \cdot \gamma < (H'_X- E_C) \cdot \gamma = 0,
\end{equation*}
which means that $\gamma$ is contained in the base locus of the linear system $|H'_X- 2E_C|$.
The natural exact sequences
\begin{equation*}
0 \to H^0(X,I_C(H_X)) \to H^0(X,\cO_X(H_X)) \to H^0(C,\cO_C(2))
\end{equation*}
and 
\begin{equation*}
0 \to H^0(X,I^2_C(H_X)) \to H^0(X,I_C(H_X)) \to H^0(C,I_C/I^2_C(2))
\end{equation*}
and equalities
\begin{equation*}
\dim H^0(X,\cO_X(H_X)) = 14,
\qquad 
\dim H^0(C,\cO_C(2)) = 3,
\qquad 
\dim H^0(C,I_C/I^2_C(2)) = 6
\end{equation*}
(the first equality is~\eqref{eq:14}, the second is evident, 
and the third follows from~\eqref{eq:N}) 
imply that we have an inequality
\begin{equation*}
\dim|H'_X- 2E_C| \ge 13 - 3 - 6 = 4.
\end{equation*}
Since moreover $\Pic(X) = \ZZ H_X$, every hyperplane section of $X$ is irreducible, 
and since the linear system~$|H'_X- 2E_C|$ is of positive dimension,
the only possible divisorial component of its base locus is~$E_C$.
Thus, it remains to check that $E_C$ is not contracted by $|-K_{X'}|$.
This, however, easily follows from the equality~$(-K_{X'})^2 \cdot E_C = (H'_X- E_C)^2 \cdot E_C = 4$.
\hfill $\Box$
\end{proof}

The construction of Theorem~\ref{theorem-diagram} can be reversed:

\begin{theorem}\label{theorem-diagram-inverse}
Let $Q\subset \PP^4$ be a smooth quadric and let $\Gamma\subset Q$ be a smooth rational quadratically normal sextic curve.
Then there is a smooth prime Fano threefold $X$ of genus $12$ and a smooth conic $C \subset X$ related to $(Q,\Gamma)$ by the diagram~\eqref{diagram}.
\end{theorem}

\begin{proof}
Let $Q'$ be the blowup of $Q$ along $\Gamma$.
By Lemma~\ref{lemma-Z} below the linear system $|-K_{Q'}| = |3H'_Q- E_\Gamma|$ on~$Q'$ is base point free.
So, according to \cite[\S 4.1]{Iskovskikh-Prokhorov-1999} we have to show that 
the morphism defined by this linear system does not contract divisors.
The argument is analogous to that of Lemma~\ref{lemma:kxprime}.

Assume $\gamma$ is a curve on $Q'$ contracted by the linear system 
$|-K_{Q'}|$, so that~\mbox{$-K_{Q'} \cdot \gamma = 0$}.
Then $H'_Q \cdot \gamma \ne 0$, since $\gamma$ is an effective curve, so since $H'_Q$ is nef we have $H'_Q \cdot \gamma > 0$.
Therefore
$E_\Gamma \cdot \gamma > 0$ and so
\begin{equation*}
(2H'_Q- mE_\Gamma) \cdot \gamma < (3H'_Q-E_\Gamma)\cdot \gamma = 0,\qquad \text{for all $m \ge 1$}.
\end{equation*}
Take $m$ to be the maximal such that $|2H'_Q- m E_\Gamma| \neq \varnothing$ (by Remark~\ref{remark:gamma-qn} we have $m\ge 1$).
Let~$F \in |2H'_Q- m E_\Gamma|$ be any member.
Then $\gamma\subset F$.
Since $\Gamma$ does not lie in a hyperplane (Remark~\ref{remark:gamma-qn}), such $F$ is irreducible, and 
thus it remains to show that the morphism defined by the linear system $|-K_{Q'}|$ does not contract~$F$.
This follows from
\begin{equation*}
(-K_{Q'})^2 \cdot (2H'_Q- mE_\Gamma) =(3H'_Q-E_\Gamma)^2 \cdot (2H'_Q- mE_\Gamma) = 24-20m.
\end{equation*}

Now, we can make a flop $Q'\dashrightarrow X'$ and consider the Mori contraction $X'\to X$.
Solving Diophantine equations, as in \cite[\S 4.1]{Iskovskikh-Prokhorov-1999} (see also \cite{Cutrone-Marshburn})
one can show that $X'\to X$ is the blowup of a smooth conic on a prime Fano threefold $X$ of genus 12.
\hfill $\Box$
\end{proof}

\begin{lemma}
\label{lemma-Z}
Let $\Gamma\subset \PP^4$ be a smooth rational quadratically normal curve of degree $6$.
Then~$\Gamma$ is a scheme-theoretic intersection of cubics.
\end{lemma}

\begin{proof}
By Remark~\ref{remark:gamma-qn} the curve $\Gamma$ is not contained in a hyperplane.
Hence~\cite[Corollary]{Gruson-Lazarsfeld-Peskine} applies to $\Gamma$ and shows that $\Gamma$ is an intersection of cubics, 
unless it has a 4-secant line.
It remains to show that a curve~$\Gamma$ with a 4-secant line is not quadratically normal.

Indeed, let $L \subset \PP^4$ be a 4-secant line and denote by $D := L \cap \Gamma$ the scheme-theoretic intersection 
(it is a zero-dimensional subscheme in $\Gamma$ of length at least 4).
Let $I_L \subset \OOO_{\PP^4}$ and $I_D \subset \OOO_\Gamma$ be the ideal sheaves.
The space $H^0(\PP^4,I_L(2))$ of quadrics in $\PP^4$ containing~$L$ has codimension 3 in the space of all quadrics.
On the other hand, its image in the space~$H^0(\Gamma,\OOO_\Gamma(12))$ is contained in the subspace $H^0(\Gamma,I_D(12))$, 
which has codimension at least 4.
Therefore, the map $H^0(\PP^4,\OOO(2)) \to H^0(\Gamma,\OOO_\Gamma(12))$ is not surjective.
\hfill $\Box$
\end{proof}

\begin{remark}
\label{remark:functoriality}{\rm
The constructions of Theorem~\ref{theorem-diagram} and Theorem~\ref{theorem-diagram-inverse} are mutually inverse.
Moreover, these constructions are functorial.
In other words, let $\varphi \colon X_1 \to X_2$ be an isomorphism of prime Fano threefolds of genus~12 and 
let $C_1 \subset X_1$ and $C_2 = \varphi(C_1)$ be smooth conics on them.
If $(Q_i,\Gamma_i)$ is the quadric with a sextic curve associated to the pair $(X_i,C_i)$ then the isomorphism $\varphi$ extends in a unique way 
to an isomorphism of diagrams~\eqref{diagram}. 
In particular, it induces an isomorphism $\psi \colon Q_1 \to Q_2$ such that~\mbox{$\psi(\Gamma_1) = \Gamma_2$}.

Conversely, let $\psi \colon Q_1 \to Q_2$ be an isomorphism of smooth quadrics,
and let $\Gamma_1 \subset Q_1$ and $\Gamma_2 = \psi(\Gamma_1)$ be smooth rational quadratically normal sextic curves on them.
If $(X_i,C_i)$ is the prime Fano threefold of genus~12 with a conic associated to the pair $(Q_i,\Gamma_i)$ then the isomorphism $\psi$ extends in a unique way 
to an isomorphism of diagrams~\eqref{diagram}. 
In particular, it induces an isomorphism $\varphi \colon X_1 \to X_2$ such that~$\varphi(C_1) = C_2$.

In particular, if the pair $(Q,\Gamma)$ corresponds to a pair $(X,C)$, we have an isomorphism
\begin{equation}
\label{eq:iso-aut}
\Aut(X,C) \cong \Aut(Q,\Gamma),
\end{equation}
where $\Aut(X,C) \subset \Aut(X)$ is the group of automorphisms of $X$ that preserve the conic~$C$ 
and similarly~$\Aut(Q,\Gamma) \subset \Aut(Q)$ is the group of automorphisms of $Q$ that preserve the sextic~$\Gamma$.}
\end{remark}

\begin{remark}
\label{remark:families}{\rm
By using a relative version of MMP one can see that 
the constructions of Theorem~\ref{theorem-diagram} and Theorem~\ref{theorem-diagram-inverse} work in smooth families.
Namely, if $\cX \to S$ is a smooth morphism whose fibers are prime Fano threefolds of genus~$12$ 
and $\cC \subset \cX$ is a subscheme which is smooth over $S$ and whose fibers are conics, 
there is a relative over $S$ version of the diagram~\eqref{diagram} with the same description of fibers.
Similarly, if $\cQ \to S$ is a smooth morphism whose fibers are three-dimensional quadrics
and $\varGamma \subset \cQ$ is a subscheme which is smooth over $S$ and whose fibers are rational quadratically normal sextic curves,
there is a relative over $S$ version of the diagram~\eqref{diagram} and again with the same description of fibers.}
\end{remark}

\section{Automorphisms}
\label{section:automorphisms}

To apply the results of Section~\ref{section:transformation}
for a description of prime Fano threefolds $X$ of genus~12 with a $\GG_\m$-action and of their automorphism groups,
we need to show that any such~$X$ has a smooth conic that is $\Aut(X)$-invariant.
For this we will need a description of $\GG_\m$-invariant lines from~\cite{Kuznetsov-Prokhorov-Shramov}.
Recall the Mukai--Umemura threefold $X^\MU$, see~\cite{Mukai-Umemura-1983}.

\begin{lemma}
\label{lemma:lines}
Let $X$ be a prime Fano threefold of genus~$12$ with a faithful\/ $\GG_\m$-action which is not isomorphic to the Mukai--Umemura threefold $X^\MU$.
There are exactly two $\GG_\m$-invariant lines $L_1$ and $L_2$ on $X$, these lines are the only special lines on $X$, i.e., 
\begin{equation*}
\NNN_{L_i/X}\simeq \OOO_{L_i}(1)\oplus \OOO_{L_i}(-2),
\end{equation*}
and they do not meet.
\end{lemma}
\begin{proof}
By~\cite[Theorem~5.3.10]{Kuznetsov-Prokhorov-Shramov} the variety $X$ belongs to the family~$X^{\mathrm{m}}(u)$ 
that was described in~\cite[Example~5.3.4]{Kuznetsov-Prokhorov-Shramov}.
Its Hilbert scheme of lines $\Sigma(X)$ was described in~\cite[Proposition~5.4.4]{Kuznetsov-Prokhorov-Shramov} 
as the union of two smooth rational curves with two points of tangency as shown on a picture below:
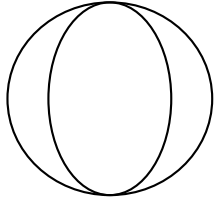
\begin{equation*}
\begin{tikzpicture}[xscale = .7, yscale = .8] 
\draw[thick] (0,0) ellipse (5em and 10ex);
\draw[thick] (0,0) ellipse (3em and 10ex);
\end{tikzpicture}
\end{equation*}
Let $L_1$ and $L_2$ be the lines on $X$ corresponding to the singular points of $\Sigma(X)$. 
Clearly, these lines are $\GG_{\m}$-invariant.
By~\cite[Corollary 2.1.6]{Kuznetsov-Prokhorov-Shramov} the lines $L_1$ and $L_2$ are special,
and they are the only special lines on $X$.
Finally, the lines $L_1$ and $L_2$ on $X$ do not meet, see \cite[Proposition~5.4.3]{Kuznetsov-Prokhorov-Shramov}.

It remains to show that the lines $L_1$ and $L_2$ are the only $\GG_\m$-invariant lines on $X$, i.e., 
that the natural $\GG_\m$-action on both components of the Hilbert scheme $\Sigma(X)$ is nontrivial.
For this we recall from~\cite[Proposition~5.4.4]{Kuznetsov-Prokhorov-Shramov}
the isomorphism of $\Sigma(X) \setminus \{[L_1]\}$ with a locally closed subset $\Sigma_Z(Y) \setminus \ell$ 
of the Hilbert scheme of lines on a Fano threefold~$Y$ of index~2 and degree~5.
This isomorphism is $\GG_\m$-equivariant by construction and the set~$\Sigma_Z(Y) \setminus \ell$ 
is described explicitly in~\cite[Lemma~5.4.1]{Kuznetsov-Prokhorov-Shramov}.
In particular, the $\GG_\m$-action on its components is indeed nontrivial.
\hfill $\Box$
\end{proof}

Now we can pass to conics.
Recall that a conic on a projective variety $X \subset \PP^N$ is a subscheme $C \subset X$ with Hilbert polynomial $p_C(t) = 1 + 2t$.
There are three types of conics --- smooth conics, i.e., curves isomorphic to the second Veronese embedding of $\PP^1$,
reducible conics, i.e., unions of two lines meeting at a point,
and non-reduced conics, i.e., non-reduced schemes $C$ such that $C_{\mathrm{red}} = L$ is a line and $I_L/I_C \cong \cO_L(-1)$,
see~\cite[Lemma~2.1.1(ii)]{Kuznetsov-Prokhorov-Shramov}.

\begin{proposition}
\label{proposition:conic}
Let $X$ be a prime Fano threefold of genus~$12$ with a faithful\/ $\GG_\m$-action which is not isomorphic to the Mukai--Umemura threefold $X^\MU$.
Then $X$ contains a smooth $\Aut(X)$-invariant conic.
\end{proposition}

In Lemma~\ref{lemma-C-unique}(ii) we will show that an $\Aut(X)$-invariant conic on such $X$ is unique.

\begin{proof}
By~\cite{Prokhorov-1990c} or~\cite[\S 5.4]{Kuznetsov-Prokhorov-Shramov} the identity component of $\Aut(X)$ is exactly the torus~$\GG_{\m}$.
Consider the Hilbert scheme $S(X)$ of conics on $X$.
We have 
\begin{equation}
\label{equation-S(X)}
S(X) \cong \PP^2
\end{equation}
by \cite{Kollar2004b} or~\cite[Proposition B.4.1]{Kuznetsov-Prokhorov-Shramov}.
The $\GG_\m$-invariant lines $L_1$ and $L_2$ on $X$ are special, hence by~\cite[Remarks~2.1.2 and~2.1.7]{Kuznetsov-Prokhorov-Shramov} 
there are non-reduced conics $C_1$ and~$C_2$ such that $(C_i)_{\mathrm{red}} = L_i$, and these are the only non-reduced conics on~$X$.
Clearly, these conics are $\GG_\m$-invariant.

Let~$S(X)^{\GG_\m} \subset S(X)$ be the fixed point locus of $\GG_{\m}\subset\Aut(X)$.
Since the natural $\GG_{\m}$-action on $S(X)$ is faithful by~\cite[Lemma~4.3.4]{Kuznetsov-Prokhorov-Shramov},
there are two possibilities: 
\begin{equation}
\label{eq:sx-gm}
\text{either $S(X)^{\GG_\m} = \{c_1,c_2,c_3\}$\qquad or $S(X)^{\GG_\m} = \{c_0\}\cup \ell$}, 
\end{equation}
where $\ell$ is a line on $S(X)=\PP^2$ and the points $c_1$ and $c_2$ correspond to the non-reduced conics $C_1$ and $C_2$.
In the first case the conic corresponding to the third point $c_3$ is $\Aut(X)$-invariant.
In the second case, if the points $c_1$ and $c_2$ lie on $\ell$, the conic corresponding to the point $c_0$ is $\Aut(X)$-invariant.
Finally, if only one of the points $c_1$ and $c_2$ lies on~$\ell$, then this point is $\Aut(X)$-invariant, 
hence the finite group $\Aut(X)/\GG_\m$ acting on~$\ell$ has an invariant point, 
hence $\Aut(X)/\GG_\m$ is a cyclic group, hence it has yet another fixed point on~$\ell$, which then corresponds to an $\Aut(X)$-invariant conic.
Thus, in all these cases we have found an~$\Aut(X)$-invariant conic~$C$ on~$X$ that is distinct from the non-reduced conics $C_1$ and~$C_2$.
The only thing left is to show that $C$ is non-singular.

Indeed, as we mentioned above $X$ has no non-reduced conics distinct from $C_1$ and $C_2$.
On the other hand, if a $\GG_\m$-invariant conic $C$ is a union of two distinct lines meeting at a point
then each of these lines is $\GG_\m$-invariant, hence $C = L_1 \cup L_2$ by Lemma~\ref{lemma:lines}.
But the $\GG_\m$-invariant lines $L_1$ and $L_2$ do not meet (again by Lemma~\ref{lemma:lines}), so this is also impossible.
\hfill $\Box$
\end{proof}

By Remark~\ref{remark:functoriality}, if $X$ is a prime Fano threefold of genus~12 with a faithful $\GG_\m$-action and~$C \subset X$ is a smooth $\Aut(X)$-invariant conic,
the corresponding pair $(Q,\Gamma)$ has a faithful $\GG_m$-action.
We show below how $\Gamma$ and $Q$ look like.

\begin{lemma}
\label{lemma:gamma}
Assume $\GG_\m$ acts faithfully on $\PP^4$ and $\Gamma \subset \PP^4$ is 
a $\GG_\m$-invariant smooth rational quadratically normal sextic curve.
Then in suitable coordinates $\Gamma$ is the image of the map~\eqref{eq:gamma}.
\end{lemma}
\begin{proof}
By Remark~\ref{remark:gamma-qn} the curve $\Gamma$ does not lie in a hyperplane.
Therefore, we can assume that the action of~$\GG_\m$ on $\Gamma$ is faithful 
and that the image of~$\GG_\m$ in~$\Aut(\Gamma) \cong \PGL_2$ is the standard torus.
Since $\Gamma$ spans $\PP^4$ the embedding $\Gamma \to \PP^4$ canonically factors as
\begin{equation*}
\Gamma \xrightarrow{\ \mathsf{v}_6\ } \PP^6 \dashrightarrow \PP^4,
\end{equation*}
where the first arrow is the Veronese embedding of degree~6 and the second arrow is a linear projection with center a line.
Since the composition is $\GG_\m$-equivariant, the line is $\GG_\m$-invariant.
If $y_0,y_1,y_2,y_3,y_4,y_5,y_6$ are the standard weight coordinates on $\PP^6$, 
the line is given by equations $y_i = 0$ for $i \in I$, where $I \subset \{0,\dots,6\}$ is a subset of cardinality~5.
Since the image of $\Gamma$ is a smooth curve, this line does not intersect the tangent lines to~$\Gamma$ at $0$ and~$\infty$,
hence the set $I$ contains $\{0,1,5,6\}$. 
If the fifth element of $I$ is~$2$ or~$4$, the curve~$\Gamma$ has a 4-tangent line and so
is not quadratically normal by the proof of Lemma~\ref{lemma-Z}.
Hence the fifth element in~$I$ is~3.
\hfill $\Box$
\end{proof}

\begin{lemma}
\label{lemma:quadrics}
Quadrics passing through the curve $\Gamma$ defined by~\eqref{eq:gamma} form a pencil generated by the quadrics~\eqref{eq:q0-q8}.
A~quadric $Q_u = u_0Q_0 + u_1Q_\infty$ from this pencil is smooth if and only if $u =u_0/u_1 \not\in \{0,1,\infty\} \subset \PP^1$.
\end{lemma}
\begin{proof}
Straightforward.
\hfill $\Box$
\end{proof}

\begin{lemma}
\label{lemma:aut-q-gamma}
If $\Gamma \subset \PP^4$ is the curve defined by~\eqref{eq:gamma} and $Q$ is a quadric 
from the pencil generated by the quadrics~\eqref{eq:q0-q8}, then $\Aut(Q,\Gamma) \cong \GG_\m \rtimes \ZZ/2\ZZ$.
\end{lemma}
\begin{proof}
Since $\Gamma$ spans $\PP^4$, we have $\Aut(Q,\Gamma) \subset \Aut(\Gamma) \cong \PGL_2$.
Furthermore, it is easy to see that the points~$(1:0:0:0:0)$ and $(0:0:0:0:1)$ are the only points on $\Gamma$ 
that lie on the singular locus of one of the quadrics passing through $\Gamma$.
Therefore, 
\begin{equation*}
\Aut(Q,\Gamma) \subset \GG_\m \rtimes \ZZ/2\ZZ,
\end{equation*}
where $\GG_\m$ is the torus that acts on $\Gamma$ preserving the above two points (by rescaling one of the coordinates~$t_0$ and $t_1$), 
and $\ZZ/2\ZZ$ is generated by the involution 
\begin{equation}
\label{eq:iota}
\iota \in \Aut(\Gamma),\qquad
\iota \colon (t_0:t_1) \longmapsto (t_1:t_0)
\end{equation}
that normalizes this torus.
On the other hand, it is easy to see that both the torus and the involution preserve any quadric in the pencil passing through $\Gamma$, hence the claim.
\hfill $\Box$
\end{proof}

Now we are ready to prove the claim of Theorem~\ref{main} concerning automorphism groups.

\begin{corollary}
\label{corollary:automorphisms}
Let $X$ be a prime Fano threefold of genus~$12$ with a faithful\/ $\GG_\m$-action which is not isomorphic to the Mukai--Umemura threefold.
Then $\Aut(X) \cong \GG_\m \rtimes \ZZ/2\ZZ$.
\end{corollary}
\begin{proof}
By Proposition~\ref{proposition:conic} there is a smooth $\Aut(X)$-invariant conic $C \subset X$, hence
\begin{equation*}
\Aut(X,C) = \Aut(X).
\end{equation*}
Let $(Q,\Gamma)$ be the smooth quadric and the sextic curve associated with $(X,C)$ by Theorem~\ref{theorem-diagram}.
By Lemma~\ref{lemma:gamma} the curve $\Gamma$ is given by~\eqref{eq:gamma} and by Lemma~\ref{lemma:quadrics} we have $Q = Q_u$ for some~$u \in \PP^1 \setminus \{0,1,\infty\}$.
Finally, by Lemma~\ref{lemma:aut-q-gamma} we have~$\Aut(Q,\Gamma) \cong \GG_\m \rtimes \ZZ/2\ZZ$.
Combining this with~\eqref{eq:iso-aut} and the above equality, we deduce the corollary.
\hfill $\Box$
\end{proof}

\section{Isomorphism classes}
\label{section:iso-classes}

Now we switch to the proof of the part of Theorem~\ref{main} describing isomorphism classes of 
prime Fano threefolds $X$ of genus~$12$ with a faithful~$\GG_\m$-action.
For this we will need a couple of observations about the action of $\Aut(X)$ on $X$.
We denote by
\begin{equation*}
\iota_X \colon X \to X
\end{equation*}
the involution corresponding to~\eqref{eq:iota} under the isomorphism~\eqref{eq:iso-aut}.
Recall that, as it was explained in the proof of Proposition~\ref{proposition:conic}, if $X$ is not the Mukai--Umemura threefold $X^\MU$, 
there are precisely two special lines $L_1$ and $L_2$ on $X$.

\begin{lemma}
\label{lemma-C-unique}
Let $X$ be a prime Fano threefold of genus~$12$ with a faithful\/ $\GG_\m$-action which is not isomorphic to the Mukai--Umemura threefold.
\begin{enumerate}
\item[\rm (i)]
The involution $\iota_X$ swaps the special lines $L_1$ and $L_2$ on $X$;
in particular $X$ has no $\Aut(X)$-invariant lines.
\item[\rm (ii)]
An $\Aut(X)$-invariant conic on $X$ is unique.
\item[\rm (iii)]
The non-reduced conics $C_1$ and $C_2$ supported on the $\GG_\m$-invariant lines~$L_1$ and~$L_2$,
and the smooth $\Aut(X)$-invariant conic $C$ are the only $\GG_\m$-invariant conics on~$X$.
\end{enumerate}
\end{lemma}
\begin{proof}
The first part of the lemma follows immediately from~\cite[Proposition~5.4.6]{Kuznetsov-Prokhorov-Shramov}, 
where it was checked that~$\Aut(X,L_1) \cong \GG_\m$.
In a combination with Lemma~\ref{lemma:lines} this proves that there are no $\Aut(X)$-invariant lines on $X$.

To prove the uniqueness of the $\Aut(X)$-invariant conic constructed in Proposition~\textup{\ref{proposition:conic}}
we have to show that the action of the involution $\iota_X$ 
on the $\GG_\m$-fixed locus $S(X)^{\GG_\m}$ in the Hilbert scheme of conics $S(X)$ has a single fixed point.
Recall that we have two possibilities~\eqref{eq:sx-gm}.

In the first case the points $c_1$ and $c_2$ correspond to the non-reduced conics supported on $L_1$ and $L_2$, 
hence they are swapped by the involution $\iota_X$.
Thus the third point $c_3$ is the only point in~$S(X)^{\GG_\m}$ fixed by~$\iota_X$.
Moreover, in this case there are exactly three $\GG_\m$-invariant conics on $X$.
So, it remains to show that the second case of~\eqref{eq:sx-gm} does not occur.

In the second case the point $c_0$ is fixed by $\iota_X$, hence the points $c_1$ and $c_2$ 
corresponding to the non-reduced conics belong to the line $\ell$.
The involution $\iota_X$ preserves $\ell$, hence has two fixed points $c'$ and $c''$ on it, which are thus $\Aut(X)$-invariant.
So, we have a triple~$(c_0,c',c'')$ of $\Aut(X)$-invariant points on $S(X)$, hence the image of $\Aut(X)$ in $\Aut(S(X)) \cong \PGL_3$ is abelian.
But the action of $\Aut(X)$ on $S(X)$ is faithful by~\cite[Lemma~4.3.4]{Kuznetsov-Prokhorov-Shramov}, so this
contradicts the fact that the conjugation by $\iota$ acts on $\GG_\m$ as inversion (see~\eqref{eq:iota}).
\hfill $\Box$
\end{proof}

Most of things we discussed so far were related to all $X$ with a faithful $\GG_\m$-action except of the Mukai--Umemura threefold $X^\MU$.
Now we note that the latter can also be covered by the same approach.
We refer to~\cite{Mukai-Umemura-1983} for a description of $X^\MU$ and of its Hilbert schemes of lines.

Recall that $\Aut(X^\MU) \cong \PGL_2$ which acts on the Hilbert scheme of conics $S(X^\MU) \cong \PP^2$ 
as on the projectivization of an irreducible representation.
In particular, there are two $\Aut(X^\MU)$-orbits on~$S(X^\MU)$.
One of them is a conic $S_\nr \subset S(X^\MU)$ that parameterizes non-reduced conics in~$X^\MU$ 
(thus~$S_\nr$ also parameterizes lines on $X^\MU$, all of which are special).
The complement~$S(X) \setminus S_\nr$ parameterizes smooth conics on $X^\MU$.
In particular for every smooth conic $C \subset X^\MU$ we have
\begin{equation*}
\Aut(X^\MU,C) \cong \GG_\m \rtimes \ZZ/2\ZZ,
\end{equation*}
the normalizer of a torus in $\PGL_2$.
So, applying the construction of Theorem~\ref{theorem-diagram} to the 
pair~$(X^\MU,C)$ we obtain a pair $(Q,\Gamma)$, 
where $\Gamma$ is the curve defined by~\eqref{eq:gamma} and~$Q$ is a smooth 
quadric from the pencil generated by~\eqref{eq:q0-q8}.
Moreover, since the group~$\Aut(X^\MU)$ acts transitively on smooth conics in~$X^\MU$, 
Remark~\ref{remark:functoriality} shows that different 
choices of a smooth conic $C$ produce the same pair $(Q,\Gamma)$.
Thus, there is a particular quadric in the pencil corresponding to the 
Mukai--Umemura threefold.
We will call it the \emph{Mukai--Umemura quadric} and denote by $Q^\MU$.
In Proposition~\ref{proposition:q-mu} below we show $Q^\MU = Q_{-1/4}$.
And meanwhile, we observe that the above argument proves the first part of 
Theorem~\ref{main}.

\begin{corollary}
\label{corollary:bijection}
Let $X$ be a prime Fano threefold of genus~$12$ with a faithful $\GG_\m$-action.
We associate with $X$ the quadric $Q = Q_u$ from the pencil~\eqref{eq:q0-q8}
produced by the construction of Theorem~\textup{\ref{theorem-diagram}} applied 
\begin{itemize}
\item to the unique $\Aut(X)$-invariant smooth conic on $X$, if $X \not\cong X^\MU$, or
\item to an arbitrary smooth conic, if $X \cong X^\MU$.
\end{itemize}
This defines a bijection between isomorphism classes of such $X$ and the set $\PP^1 \setminus \{0,1,\infty\}$ of smooth quadrics in the pencil.
\end{corollary}

\begin{remark}{\rm
Most probably, the punctured line $\PP^1 \setminus \{0,1,\infty\}$ is the coarse moduli space 
for prime Fano threefolds of genus~$12$ with a faithful $\GG_\m$-action.
This should follow from our results in view of Remark~\ref{remark:families}.}
\end{remark}

To finish the proof of Theorem~\ref{main} it remains to identify the Mukai--Umemura quadric.

\begin{proposition}
\label{proposition:q-mu}
The quadric $Q^\MU \subset \PP^4$ associated with a pair $(X^\MU,C)$, where $X^\MU$ is the Mukai--Umemura threefold
and $C$ is any smooth conic on $X^\MU$ is the quadric 
\begin{equation*}
Q_{(-1:4)} = \{ -y_0y_6- 4y_1y_5 + 5y_3^2 = 0\}.
\end{equation*}
\end{proposition}
\begin{proof}
Recall the constructions of the Mukai--Umemura threefold $X = X^\MU$ from~\cite{Mukai-Umemura-1983}.
Let $M_d$ denote the vector space of degree $d$ homogeneous polynomials in two variables $x$ and~$y$ with its natural $\GL_2$-action.
Then $X$ can be realized inside $\PP(M_{12} \oplus M_0) \cong \PP^{13}$ as the closure of a $\PGL_2$-orbit:
\begin{equation*}
X = \overline{ \PGL_2 \cdot (\phi_{12},1) },
\qquad
\text{where $\phi_{12} = xy(x^{10} + 11x^5y^5 + y^{10})$}.
\end{equation*}
We consider the standard torus $\GG_\m \subset \PGL_2$ and set
\begin{equation}
\label{eq:conic-mu}
C = \overline{ \GG_\m \cdot (\phi_{12},1) } =
\overline{\{(xy(t^{-1}x^{10} + 11x^5y^5 + ty^{10}),1) \mid t \in \CC^\times \}} \subset \PP(M_{12} \oplus M_0),
\end{equation}
to be the closure of the $\GG_\m$-orbit of the point $(\phi_{12},1) \in X$.
Clearly, $C$ is a smooth $\GG_\m$-invariant conic on~$X$.
We apply the construction of Theorem~\ref{theorem-diagram} to the pair $(X,C)$.

Consider the linear functions $z_i$ on $\PP(M_{12} \oplus M_0)$ that take~$(f,c) \in M_{12} \oplus M_0$ 
to the coefficient of $f$ at~$x^{12-i}y^i$ (then its $\GG_\m$-weight is $i-6$).
Let also $\bar z$ be the linear function that takes $(f,c)$ to $c$ (its weight is 0).
We claim that the double projection $\xi: X \dashrightarrow Q\subset \PP^4$ of Theorem~\ref{theorem-diagram} 
is defined (in appropriate coordinates on $\PP^4$) by the map
\begin{equation}
\label{eq:map-explicit}
(f,c) \longmapsto \bigl(z_3(f,c),\, z_4(f,c),\, z_6(f,c)- 11\bar{z}(f,c),\, z_8(f,c),\, z_9(f,c)\bigr).
\end{equation}
Indeed, we know that the map should be $\GG_\m$-equivariant, 
and we know from Lemma~\ref{lemma:gamma} that the weights of~$\GG_\m$ on $\PP^4$ should be $(-3,-2,0,2,3)$.
Clearly, the functions $z_3$, $z_4$, $z_8$, and $z_9$ are the only linear functions of weights $-3$, $-2$, $2$, and $3$ respectively.
On the other hand, there are two functions $z_6$ and $\bar z$ of weight zero, so we should take an appropriate linear combination of those.
But this function should vanish at the point $(\phi_{12},1) \in C$, and so it remains to note that $z_6(\phi_{12},1) = 11$, while~$\bar z(\phi_{12},1) = 1$.

To find the quadric $Q$ and the curve $\Gamma$ we apply the map~\eqref{eq:map-explicit} to the boundary divisor 
\begin{equation*}
D = \overline{ \PGL_2 \cdot\, (xy^{11},0) } = X \cap \PP(M_{12}) =
X \setminus \bigl( \PGL_2\cdot\, (\phi_{12},1) \bigr) \subset X.
\end{equation*}
This is an anticanonical divisor; furthermore it is a non-normal surface, 
whose normalization is isomorphic to $\PP(M_1) \times \PP(M_1)$ via the map 
\begin{equation*}
\nu \colon \PP(M_1) \times \PP(M_1) \longrightarrow D,
\quad 
((a_0x +a_1y), (b_0x + b_1y)) \longmapsto (a_0x + a_1y)(b_0x + b_1y)^{11},
\end{equation*}
and whose singular locus is the image of the diagonal in $\PP(M_1) \times \PP(M_1)$ 
(see \cite[Lemmas 1.6, 6.1 and Remark 6.5]{Mukai-Umemura-1983}).
Restricting to the affine chart $\mathbb{A}^2 \subset \PP(M_1) \times \PP(M_1)$ defined by
$a_0 \ne 0$, $b_0 \ne 0$ with coordinates $a = a_1/a_0$, $b = b_1/b_0$, and composing the map~$\nu$ with
$\xi: X \dashrightarrow Q\subset \PP^4$ (see~\eqref{eq:map-explicit}), we find
\begin{equation}
\label{eq:surface}
\begin{split}
z_3\bigl((x + ay)(x + by)^{11},\, 0\bigr) &= \hphantom{0}55(a + 3b)b^2,\\
z_4\bigl((x + ay)(x + by)^{11},\, 0\bigr) &= 165(a + 2b)b^3,\\
z_6\bigl((x + ay)(x + by)^{11},\, 0\bigr) &= 462(a + b)b^5,\\
z_8\bigl((x + ay)(x + by)^{11},\, 0\bigr) &= 165(2a + b)b^7,\\
z_9\bigl((x + ay)(x + by)^{11},\, 0\bigr) &= \hphantom{0}55(3a + b)b^8
\end{split}
\end{equation} 
(here $55$, $165$ and $462$ are the binomial coefficients $\binom{11}2$, $\binom{11}3$, and $\binom{11}{5}$).
Denote by $D'$ the closure in $\PP^4$ of the image of $\mathbb{A}^2$ under the map~\eqref{eq:surface}.
It is easy to check that the quadric
\begin{equation}
\label{eq:quadric-mu}
1764z_3z_9- 784z_4z_8 + 125z_6^2 = 0
\end{equation}
is the only quadric containing $D'$. Moreover, denoting the golden ratio by
\begin{equation*}
\upphi = \frac{1 + \sqrt{5}}2,
\end{equation*}
one easily checks that for any $t \in \mathbb{G}_\m$ one has
\begin{multline}
\label{eq:gamma-1}
\xi (\nu (-\upphi^2 t, t )) = \\ 
\Big(5(3 - \upphi^2) : 15(2 - \upphi^2)t : 42(1 - \upphi^2)t^3 : 15(1 - 2\upphi^2)t^5 : 5(1 - 3\upphi^2)t^6 \Big) = \\
\xi (\nu(t, -\upphi^2 t ))
\end{multline}
(for the first equality one has to rescale each coordinate by $t^3$ and for the second by~$-\upphi^{10}t^3$).
It follows that the two irreducible components of the curve in $\PP(M_1) \times \PP(M_1)$ given (in the affine coordinates) by the equation
\begin{equation*}
a^2 + 3ab + b^2 = (a + \upphi^2b)(a + \upphi^{-2}b) = 0
\end{equation*}
are mapped bijectively onto the curve in $D'$, given by the middle line of~\eqref{eq:gamma-1}.

Since the normalization $\nu \colon \PP(M_1) \times \PP(M_1) \to D$ is bijective,
this curve must be contained in the indeterminacy locus of the map $\xi^{-1}$.
The description of $\xi$ in Theorem~\ref{theorem-diagram} shows that~this curve should
coincide with $\Gamma$. 
Comparing \eqref{eq:gamma-1} with~\eqref{eq:gamma}
we find the relation between the weight coordinates $z_i$ and $y_i$:
\begin{align*}
z_3 &= \hphantom{0}5(3 - \hphantom{0}\upphi^2)y_0,	\\
z_4 &= 15(2 - \hphantom{0}\upphi^2)y_1,	\\
z_6 &= 42(1 - \hphantom{0}\upphi^2)y_3,	\\
z_8 &= 15(1 - 2\upphi^2)y_5,	\\
z_9 &= \hphantom{0}5(1 - 3\upphi^2)y_6.
\end{align*}
Substituting this into~\eqref{eq:quadric-mu} and canceling the common factors, we deduce
\begin{equation*}
Q^\MU = -Q_0 + 4Q_\infty, 
\end{equation*}
i.e. $Q^\MU = Q_{(-1:4)}$.
\hfill $\Box$
\end{proof}

\begin{remark}{\rm
One can check that the surface $D'$ is given by two equations: \eqref{eq:quadric-mu} and 
\begin{equation*}
32 z_4^2 z_6 z_8^2-630 z_3^2 z_8^3+81 z_3 z_4 z_6 z_8 z_9-630 z_4^3 z_9^2+2187 z_3^2 z_6 z_9^2=0.
\end{equation*}
Its singular locus consists of two lines
$\{z_9=z_8=z_6=0\}$, $\{z_6=z_4=z_3=0\}$,
the rational sextic curve \eqref{eq:gamma-1},
and another rational sextic curve
\begin{equation*}
(t_0 : t_1) \longmapsto (z_3 : z_4 : z_6 : z_8 : z_9) = (20t_0^6 : 45t_0^5t_1 : 84t_0^3t_1^3 : 45t_0t_1^5 : 20t_1^6)
\end{equation*}
which is the image of the diagonal of $\PP(M_1) \times \PP(M_1)$.
The singularities of $D'$ along the last curve are cuspidal.}
\end{remark}

This computation completes the proof of Theorem~\ref{main}.

\section{Concluding remarks}
\label{section:conclusion}

To finish the paper we provide some extra details for the birational transformations
of Theorems~\ref{theorem-diagram} and~\ref{theorem-diagram-inverse}.

\begin{proposition}
\label{proposition-flop}
Let $X$ be a prime Fano threefold of index~$12$ with a faithful $\GG_\m$-action 
and let $C \subset X$ be the smooth $\GG_\m$-invariant conic on $X$.
The flopping locus of the map~$\chi$ in~\eqref{diagram} is the union of the strict transforms in~$X'$ 
of the two $\GG_\m$-invariant lines $L_1$ and~$L_2$,
and the flop is given by Reid's pagoda.
\end{proposition}
\begin{proof}
Since the Hilbert scheme of conics on $X$ is smooth (see \eqref{equation-S(X)}),
the normal bundle of the smooth $\GG_\m$-invariant conic $C$ (and, in fact, of any smooth conic on $X$) is
\begin{equation}
\label{eq:N}
\NNN_{C/X}\simeq \OOO_{C}(a)\oplus \OOO_{C}(-a),
\qquad \text{$a=0$ or $1$} 
\end{equation}
(see \cite[Lemma 2.1.4, Corollary 2.1.6]{Kuznetsov-Prokhorov-Shramov}).
Therefore by~\cite[Proposition 4.4.1]{Iskovskikh-Prokhorov-1999} the flopping locus of $\chi$ in~\eqref{diagram} 
consists of the strict transforms of lines meeting $C$ and strict transforms of smooth conics meeting~$C$ at two points,
all of which are automatically $\GG_\m$-invariant.

If $X \not\cong X^\MU$ the conic $C$ is the unique smooth $\GG_\m$-invariant conic and $L_1$, $L_2$ are 
the only $\GG_\m$-invariant lines on $X$ by Lemma~\ref{lemma-C-unique}.
Therefore, the flopping locus of the map~$\chi \colon X' \dashrightarrow Q'$ is the union of the strict transforms 
of the lines $L_1$ and $L_2$.
Note that the lines $L_1$ and $L_2$ do not meet by Lemma~\ref{lemma:lines}.

If $X \cong X^\MU$ the explicit description of the Hilbert schemes of conics and lines shows 
that $C$ is still the unique smooth $\GG_\m$-invariant conic and there are precisely two $\GG_\m$-invariant lines on $X$;
for instance, if the curve $C$ is given by~\eqref{eq:conic-mu}, these lines are
\begin{equation*}
L_1 = \{ (a_1x + a_2y)x^{11} \}
\qquad \text{and} \qquad 
L_2 = \{ (a_1x + a_2y)y^{11} \}.
\end{equation*}
Clearly, these lines do not meet.

Note that in both cases $L_1 \cap L_2 = \varnothing$ and the lines $L_1$ and $L_2$ are special, 
hence the normal bundles of their strict transforms in $X'$ are isomorphic to $\cO \oplus \cO(-2)$,
hence the flop $\chi$ is given by Reid's pagoda.
\hfill $\Box$
\end{proof}

In the following remark we describe the flopping locus and the exceptional divisor 
of the rational map~\mbox{$\xi^{-1} \colon Q \dashrightarrow X$}.

\begin{remark}\label{remark-description-F}{\rm
The base locus of the pencil of quadrics $\langle Q_0, Q_\infty \rangle$ is a surface $F \subset \PP^4$ of degree~$4$.
This surface is the exceptional divisor of the map $\xi^{-1} \colon {} Q \dashrightarrow X$ by Theorem~\ref{theorem-diagram}\,{\rm (ii)}.
The singular locus of $F$ consists of four ordinary double points
\begin{align*}
& P_0 := (1 : 0 : 0 : 0 : 0), \quad && P_6 := (0 : 0 : 0 : 0 : 1),\\
& P_1 := (0 : 1 : 0 : 0 : 0), \quad && P_5 := (0 : 0 : 0 : 1 : 0),
\end{align*}
and $F$ contains exactly four lines 
\begin{equation*}
\ell_{0,1} = \langle P_0, P_1 \rangle,\qquad 
\ell_{1,6} = \langle P_1, P_6 \rangle,\qquad 
\ell_{6,5} = \langle P_6, P_5 \rangle,\qquad 
\ell_{5,0} = \langle P_5, P_0 \rangle. 
\end{equation*}
The curve $\Gamma$ passes through $P_0$ and $P_6$ and does not pass through $P_1$ and $P_5$.
The lines~$\ell_{0,1}$ and $\ell_{6,5}$ are $3$-tangent to the curve~$\Gamma$.
The surface $F$ can be realized as the quotient of $\PP^1\times \PP^1$ by the involution $(w_1,w_2) \mapsto (-w_1,-w_2)$
with the quotient map $\PP^1\times \PP^1 \to F\subset \PP^4$ given by 
\begin{equation*}
(w_1, w_2) \longmapsto (y_0 : y_1 : y_3 : y_5 : y_6) = (1 : w_1^2 : w_1w_2 : w_2^2 : w_1^2w_2^2).
\end{equation*}
In particular, $F$ is a toric del Pezzo surface with $\Pic(F) \cong \ZZ^2$ and $\operatorname{Cl}(F) \cong \ZZ^2 \oplus \ZZ/2$. 
One can also realize~$F$ as the anti-canonical image of the blowup~$F'$ of $\PP^1\times \PP^1$ 
at four points $(0,0)$, $(0,\infty)$, $(\infty,\infty)$, $(\infty,0)$.

The two $3$-tangent lines $\ell_{0,1}$ and $\ell_{6,5}$ are contained in the flopping locus for $\chi^{-1}$.
Since by Proposition~\ref{proposition-flop} the flopping locus consists of two irreducible components, 
the strict transforms in~$Q'$ of these two lines form the whole flopping locus of $\chi^{-1}$.}
\end{remark}

Finally, we give a description of one of the boundary points in the family $\PP^1 \setminus \{0,1,\infty\}$ 
of $\GG_\m$-invariant threefolds $X$.
Note that the quadric $Q_1=\{y_0y_6-y_1y_5=0\}$ is a cone over a smooth quadric surface and contains $\Gamma$.
Its vertex 
\begin{equation*}
P_3 = (0:0:1:0:0)
\end{equation*}
is an ordinary double point on $Q_1$ that lies away from the surface $F$ described in Remark~\xref{remark-description-F}.
In particular, the point $P_3$ lies away from $\Gamma$ and away from the flopping lines of $\chi^{-1}$.
Therefore, the construction of Theorem \xref{theorem-diagram-inverse} can be applied to $(Q_1,\Gamma)$ 
and produces a prime Fano threefold $X^\m(1)$ of genus~$12$ with one ordinary double point and a faithful~$\GG_\m$-action.

On the other hand, Fano threefolds of genus~12 with a single ordinary double point were classified in~\cite{Prokhorov-nodal}.
In the next proposition we identify $X^\m(1)$ with a threefold of type~(IV);
a detailed description of such threefolds is given in~\cite[\S 7]{Prokhorov-nodal}.

\begin{proposition}
The threefold $X^\m(1)$ is a threefold of type \textup{(IV)} from~\cite[Theorem~1.2]{Prokhorov-nodal}. 
\end{proposition}

\begin{proof}
A general plane $\Pi_\lambda$ in the pencil 
\begin{equation*}
\{y_5=\lambda y_0,\, y_6=\lambda y_1\}
\end{equation*}
is contained in $Q_1$,
meets $\Gamma$ at $5$ distinct points, passes through $P_3$, and does not meet the flopping locus of~$\chi^{-1}$
(see Remark~\ref{remark-description-F}).
Let $\Pi_\lambda'\subset Q'_1$
be the strict transform of $\Pi_\lambda$ (it is isomorphic to the blowup of $\Pi_\lambda$ at 5 points).
The composition~$\sigma_X\circ \chi^{-1}$ is regular near~$\Pi_\lambda'$ and
it is easy to see that \mbox{$\deg (\sigma_X(\chi^{-1}(\Pi_\lambda'))) = 5$}
(the map $\sigma_X\circ \chi^{-1}$ contracts the strict transform of the conic~$\Pi_\lambda \cap Q_0$).
Therefore, \mbox{$\sigma_X(\chi^{-1}(\Pi_\lambda'))$} generates a pencil of quintic del Pezzo surfaces 
on~$X^\m(1)$ with a base point $P_3$ which is singular on~$X^\m(1)$. 
Blowing up the Weil divisor~$\sigma_X(\chi^{-1}(\Pi_\lambda'))$ (this blowup is a small resolution of $X^\m(1)$)
we obtain a base point free pencil of quintic del Pezzo surfaces. 
By~\cite[Theorem~1.2]{Prokhorov-nodal} this property characterizes the case~(IV).
\hfill $\Box$
\end{proof}

It would be interesting to understand degenerations corresponding to the other two boundary points 
of the family $\PP^1 \setminus \{0,1,\infty\}$ of $\GG_\m$-invariant threefolds $X$.

\section*{Acknowledgements}

The authors were partially supported by the Russian Academic Excellence Project ``5-100'', 
by RFBR grant 15-01-02164, and Program of the Presidium of the Russian Academy of Sciences No.~01
``Fundamental Mathematics and its Applications'' under grant PRAS-18-01.
A.K was also supported by the Simons foundation.
Y.P. was also supported by RFBR grant 15-01-02158.

The authors are grateful to Costya Shramov for useful discussions, 
to Izzet Coskun for explaining the proof of Lemma~\ref{lemma-Z},
and to the referee for comments.

\providecommand{\bysame}{\leavevmode\hbox to3em{\hrulefill}\thinspace}
%
%

\bibliographystyle{amsalpha}
\bibliographymark{References}
\def\cprime{$'$}

\end{document}